\documentclass[12pt]{amsart}
\usepackage{amscd}
\usepackage{pstcol,pst-plot,pst-3d}
\def\NZQ{\Bbb}               
\def\NN{{\NZQ N}}

\def\ZZ{{\NZQ Z}}

\def\D{{\Delta}}
\def\mm{{\frak m}}
\def\F{{\mathcal F}}
\def\H{{\mathcal H}}
\def\G{{\mathcal G}}

\def\a{{\bold a}}
\def\b{{\bold b}}

\def\e{{\bold e}}
\def\1{{\mathbf 1}}
\def\0{{\mathbf 0}}
\def\opn#1#2{\def#1{\operatorname{#2}}} 
\opn\dist{dist} 
\opn\diam{diam} 
\opn\depth{depth}
\opn\supp{supp}
\opn\lk{lk}
\opn\st{st}
\opn\height{ht}
\opn\N{{\mathcal N}}
\opn\F{{\mathcal F}}

\newtheorem{Theorem}{Theorem}[section]
\newtheorem{Lemma}[Theorem]{Lemma}
\newtheorem{Corollary}[Theorem]{Corollary}
\newtheorem{Proposition}[Theorem]{Proposition}
\newtheorem{Remark}[Theorem]{Remark}

\newtheorem{Example}[Theorem]{Example}

\begin{document}
\title{On the associated primes and the depth\\ of the second power \\ of squarefree monomial ideals}

\author{Naoki Terai}
\address{Department of Mathematics, Faculty of Culture
and Education, Saga University, Saga 840--8502, Japan.} 
\email{terai@cc.saga-u.ac.jp}

\author{Ngo Viet Trung}
\address{Institute of Mathematics, Vietnam Academy of Science and Technology, 18 Hoang Quoc Viet, Hanoi, Vietnam}
\email{nvtrung@math.ac.vn}


\thanks{Part of the paper was written when the authors participated to a SQuaRE program at AIM, Palo Alto, in October 2012. 
The authors would like to thank AIM for its support and hospitality. The second author is supported by the National Foundation of  
Science and Technology Development.} 

\maketitle

\begin{abstract}
We present combinatorial characterizations for the associated primes of the second power 
of squarefree monomial ideals and criteria for this power to have positive depth or depth greater than one.
\end{abstract}

\section*{Introduction}

Let $I$ be a squarefree monomial ideal in a polynomial ring $R$ over a field.
We are interested in combinatorial characterizations of the associated primes and of the depth of $I^t$, $t \ge 2$.
A squarefree monomial ideal can be viewed either as the edge ideal or the cover ideal of a hypergraph 
or as the Stanley-Reisner ideal of a simplicial complex. The problem is to describe the associated primes and the depth of $I^t$ in terms of the associated hypergraph or simplicial complex. There have been many works on the asymptotic behavior of the associated primes and the depth of $I^t$ for $t$ large enough (see e.g.  \cite{CMS}, \cite{FHV}, \cite{FHV2}, \cite{HH}, \cite{HM},  \cite{MMV}) but {\em not much is known for a fixed power} $I^t$.  
So far one could only characterize the associated primes of $I^2$ for  the case $I$ is the cover ideal of a graph \cite{FHV1}.
There was also a subtle description of the associated primes of $I^t$ for the cover ideal of a hypergraph  in \cite{FHV2}.
But this description is not formulated directly in terms of the given hypergraph.  
The depth of $I^t$  (even of $I^2$) hasn't been characterized until now except for the case $\depth R/I^t = \dim R/I^t$, i.e.~$I^t$ is Cohen-Macaulay \cite{MT1}, \cite{MT2}, \cite{RTY1}, \cite{RTY2}, \cite{TT}.
Note that the characterization of the associated primes of $I^t$ can be reduced to the case $\depth R/I^t = 0$,
when the maximal homogeneous ideal is an associated prime of $I^t$. \par

It is well known that the depth can be characterized by means of local cohomology modules. 
For a monomial ideal, one can use Takayama's formula \cite{Ta} which expresses the local cohomology modules of 
the factor ring in terms of certain simplicial complexes.
Our novel finding is that these complexes can be described in terms of the associated primes of the ideal.
We shall see below that this approach works best for the depth of the second power of a squarefree monomial ideal.
The cases of higher powers are more complicated, and we will not deal with them in this paper.
However, we believe that our approach would provide a systematic method to study the associated primes and the depth of any power
of a square free monomial ideal.
\par

Let $I$ be the edge ideal of a hypergraph $\H$.  
Let $\widetilde{I^2}$ denote the saturation of $I^2$.
We show that every monomial of $\widetilde{I^2} \setminus I^2$
corresponds to a distinguished subset of the vertex set, which we call a 2-saturating set. 
From this it follows that $\depth R/I^2 > 0$ if and only if $\H$ has no 2-saturating set (Theorem \ref{positive depth}).
A 2-saturating set is closely related to a special triangle of $\H$, a notion in hypergraph theory 
which generalizes the notion triangle of a graph.
If $\H$ is a graph, a 2-saturating set is nothing else but a dominating triangle, where 
a set $U$ of vertices is called dominating if every vertex of $\H$ is adjacent to at leat one vertex of $U$. 
As a consequence, $\depth R/I^2 > 0$ if and only if the graph has no dominating triangle (Theorem \ref{depth > 0 graph}). \par

Using the above result we can characterize an associated prime of $I^2$ combinatorially as
a cover of $\H$ on which the induced subhypergraph has a 2-saturating set (Theorem \ref{asso}). 
We also show how to find such covers and  describe 
the covers which correspond to the embedded associated primes of $I^2$.
As an application we give a combinatorial criterion for $I^{(2)} = I^2$, 
where $I^{(2)}$ denotes the second symbolic power of $I$. This criterion is different than 
the criterion for $I^{(2)} = I^2$ found by Rinaldo, Terai and Yoshida in \cite{RTY2}. 
If $\H$ is a graph, we can characterize an associated prime of $I^2$ combinatorially as a minimal cover 
or a cover which is minimal among the covers containing the neighborhood of a triangle of $\H$ (Theorem \ref{graph asso}). 
This provides a simple way to compute all associated primes of $I^2$. \par

We also show that $\diam \D^{(1)} \le 2$ if $\depth R/I^2 > 1$, 
where $\D^{(1)}$ denotes the one-dimensional skeleton of the simplicial complex whose Stanley-Reisner ideal is $I$
(Theorem \ref{depth > 1}). By a recent result of Rinaldo, Terai and Yoshida \cite{RTY2}, this condition is a criterion for 
$\depth R/I^{(2)} > 1$.  We couldn't find  a criterion for $\depth R/I^2 > 1$ in general. 
However, if $\H$ is a graph, this can be done in terms of the complementary graph $\overline \H$ of $\H$, namely,
$\depth R/I^2 > 1$ if and only if $\diam \overline \H \le 2$ and the induced graph of $\overline \H$ on the complement of the neighborhood of every triangle of $\H$ has at least two vertices and is connected (Theorem \ref{depth > 1 graph}). \par

In the case $I$ is the cover ideal of a graph, our method immediately yields a beautiful result of Francisco, Ha and Van Tuyl \cite{FHV1}, which characterizes the associated primes of $I^2$ combinatorially as edges and induced odd cycles.
Furthermore, we are able to give a combinatorial criterion for $\depth R/I^2 > 1$ 
in terms of forbidden substructures of the graph (Theorem \ref{depth > 1 cover}).  
As a consequence, we prove that $I$ is a complete intersection if $R/I^2$ satisfies Serre's condition $(S_2)$.
This gives a positive answer to a question of Rinaldo, Terai and Yoshida in \cite[Question 3.1]{RTY2}.
\par

The paper is organized as follows. In Section 1 we recall Takayama's lemma 
and discuss the consequences for the computation of the depth of an arbitrary monomial ideal. 
In Section 2 we study the condition $\depth R/I^2 > 0$. 
Section 3 is devoted to the description of the associated primes of $I^2$. 
In Section 4 we study the condition $\depth R/I^2  > 1$. 
The paper concludes with Section 5 where we apply our method to study the second power of the cover ideal of a graph. \par

After the submission of the paper the authors were informed by J. Herzog and T. Hibi that they independently obtained Theorem 2.8 in a recent preprint \cite{HH2}.

\section{Vanishing of local cohomology modules}

Let $I \neq 0$ be a monomial ideal in the polynomial ring $R = k[x_1,...,x_n]$ over a field $k$. 
Then $S/I$ is an $\NN^n$-graded algebra.
For every multidegree $\a \in \ZZ^n$ and $i \ge 0$ we denote by $H_\mm^i(R/I)_\a$ the $\a$-component of the $i$-th local cohomology module $H_\mm^i(S/I)$ of $S/I$ with respect to the maximal homogeneous ideal $\mm$ of $S$. \par

Inspired by a result of Hochster in the squarefree case \cite{Ho}, Takayama \cite{Ta} showed that $H_\mm^i(S/I)_\a$ is strongly related to the reduced homology $\tilde H_j(\D_\a(I),k)$ of a simplicial complex $\D_\a(I)$ on the vertex set $[n] = \{1,...,n\}$, which is defined as follows. \par

Let $\a = (a_1,...,a_n)$. Put $G_\a = \{i \in [n] \mid a_i < 0\}$ and $x^\a = x_1^{a_1} \cdots x_n^{a_n}$. Then
$$\D_\a(I)  := \{F \setminus G_\a \mid G_\a \subseteq F \subseteq [n],\  x^\a \not\in IR_F\},$$
where $R_F = R[x_i^{-1}\mid i \in F]$. This definition of $\D_\a(I)$ is taken from \cite [Lemma 1.2]{MT2} (see also \cite{MT1}), which is simpler than the original definition in \cite{Ta}. \par

For a set $F \subseteq [n]$ we denote by $x^F$ the monomial $\prod_{i\in F}x_i$.
For a simplicial complex $\D$ on the vertex set $[n]$ we denote by $I_\D$ the ideal of $R$ generated by the monomials $x^F$, $F \not\in \D$. 
This ideal is called the Stanley-Reisner ideal of $\D$.
Let $\D(I)$ denote the simplicial complex such that $\sqrt{I}$ is the Stanley-Reisner ideal of $\D(I)$. 
For $j = 1,...,n$, let $\rho_j(I)$ denote the maximum of positive $j$th coordinates of all vectors $\b \in \NN^n$ such that $x^\b$ is a minimal generator of $I$. 

\begin{Theorem} \label{Takayama} \cite[Theorem 1]{Ta}
$$\dim_kH_\mm^i(R/I)_\a = 
\begin{cases}
\dim_k\widetilde H_{i-|G_\a|-1}(\D_\a(I),k) & \text{\rm if }\ G_\a \in \D(I)\ \text{\rm and}\\
&\ \ \  \ a_j  < \rho_j(I)\ \text{\rm for}\ j=1,...,n,\\
0 & \text{\rm else. }
\end{cases} $$
\end{Theorem}

\begin{Remark} \label{initial}
{\rm We always have $\D_\a(I) \subseteq \D(I)$ \cite[Lemma 1.3]{MT1}.  
Moreover,  if we denote by $\0$ the vector of $\NN^n$ whose components are all zero and by $\e_1,...,\e_n$ the unit vectors of $\NN^n$, then
$\D_\0(I) =  \D_{\e_i}(I) = \D(I)$ for $i =1,...,n$ \cite[Example 1.4]{MT1}.}
\end{Remark}

For a homogeneous ideal $J$ in a polynomial ring $S$  let 
$$\widetilde J := \bigcup_{t \ge 1}(J :S^t_+),$$
where $S_+$ denotes the maximal homogenous ideal of $S$. 
The ideal $\widetilde J$ is called the {\it saturation} of $J$.
Using this notion we can describe the facets of  $\D_\a(I)$ as follows. \par

For a fixed $F \subseteq [n]$ let $S = k[x_i|\ i \not \in F]$ and $I_F = IR_F \cap S$. 
Then $I_F$ is the ideal of  $S$ obtained from the monomials of $I$ by setting $x_i = 1$ for all $i \in F$. 
Therefore, $x^\a \not\in  IR_F$ if and only if $x^{\a_F} \not\in I_F$, 
where $\a_F$ denotes the vector obtained from $\a$ by setting $a_i = 0$ for all $i \in F$.  \par

For a complex $\D$ we denote by $\F(\D)$ the set of the facets of $\D$.

\begin{Lemma}\label{facet}
 $\F(\D_\a(I))  := \{F \setminus G_\a \mid F \supseteq G_\a,\  x^{\a_F} \in \widetilde {I_F} \setminus I_F\}.$
\end{Lemma}

\begin{proof}	
By the definition of $\D_\a$,  $F \setminus G_\a \in \F(\D_\a(I))$ if and only if 
$x^\a \not\in  IR_F$ and $x^\a \in  IR_{F \cup \{i\}}$ for all $i \not\in F$.  
Since we can replace $x^\a$ by $x^{\a_F}$,  $F \setminus G_\a \in \F(\D_\a(I))$ if and only if 
$x^{\a_F} \not\in  I_F$ and $x^{\a_F} \in  IR_{F \cup \{i\}} \cap S$ for all $i \not\in F$.  
Note that $R_{F \cup \{i\}} = R_F[x_i^{-1}]$. Then $IR_{F \cup \{i\}} \cap S$ consists of elements 
of $S$ for which we can find a power $x_i^t$ such that their product belongs to $I_F$.
In other word, $IR_{F \cup \{i\}} \cap S = \bigcup_{t \ge 1}(I_F :x_i^t)$.
Since $S_+$ is generated by the variables $x_i$, $i \not\in F$, it is easy to check that
$$\bigcap_{i \not\in F}\big(\bigcup_{t \ge 1}(I_F :x_i^t)\big) = \bigcup_{t \ge 1}\big(I_F: S_+^t\big) = \widetilde {I_F}.$$
Thus, $F \in \F(\D_\a(I))$ if and only if $x^{\a_F} \in \widetilde {I_F} \setminus I_F$.
\end{proof}

\begin{Corollary}\label{empty}
Let $S = k[x_i|\ i \not\in G_\a]$ and $J = IR_{G_\a} \cap S$.
Let $\a_+$ denote the vector obtained from $\a$ by replacing every negative component by 0.
Then $\D_\a(I) = \{\emptyset\}$ if and only if $x^{\a_+} \in  \widetilde J \setminus J$.
\end{Corollary}

\begin{proof}	
$\D_\a(I) = \{\emptyset\}$ means that $\emptyset \in \F(\D_\a(I))$. 
By Lemma \ref{facet}, $\emptyset \in \F(\D_\a(I))$ if and only if $\a_{G_\a} \in \widetilde {I_{G_\a}} \setminus I_{G_\a}$.
Since $\a_+ = \a_{G_\a}$ and $J = I_{G_\a}$, this implies the assertion.
\end{proof}

We are interested in the case $\D_\a(I) = \{\emptyset\}$ because $\widetilde H_{-1}(\{\emptyset\},k) = k$.
By Theorem \ref{Takayama}, this implies $H_\mm^i(R/I)_\a \neq 0$ for $i = |G_\a|$ if $\D_\a(I) = \{\emptyset\}$. \par

\begin{Remark}
{\rm Lemma \ref{facet} shows that in order to compute the complexes $\D_\a(I)$ we need to know the associated primes of $I$. 
In fact, $\widetilde{I_F} \neq I_F$ if and only if $\widetilde{IR_F} \neq IR_F$.
If we denote by $P_F$ the ideal generated by the variables $x_i$, $i \in F$, 
then $\widetilde{IR_F} \neq IR_F$ if and only if $P_F$ is an associated prime of $I$.
Therefore, $\D_\a(I)$ is the link of $G_\a$ in a simplicial complex whose facets $F$ correspond to a set of associated primes of $I$. }
\end{Remark}

Now we will discuss some consequences of Theorem \ref{Takayama} for the vanishing of the local cohomology modules of $R/I$.\par

It is well-known that $H_\mm^0(R/I) = \widetilde I/I$ and that $\depth R/I > 0$ if and only if $H_\mm^0(R/I) = 0$ or, equivalently, $\widetilde I = I$.   We have $H_\mm^0(R/I)_\a  = \widetilde H_{-|G_\a|-1}(\D_\a(I),k)$.
Therefore, $H_\mm^0(R/I)_\a \neq 0$ if  and only if $|G_\a| = 0$ and $\D_\a(I) = \{\emptyset\}$.
By Corollary \ref{empty}, this is equivalent to  the condition $\a \in \NN^n$ and $x^\a \in \widetilde I \setminus I$, 
which complies with the well-known fact that $H_\mm^0(R/I) = \widetilde I/I$. 
So Theorem \ref{Takayama} does not give us much information on the vanishing of $H_\mm^0(R/I)$ 
except that $H_\mm^0(R/I)_\a = 0$ if $\a$ has a component $a_j \ge \rho_j(I)$. \par

On the other hand, we have the following combinatorial criterion for the vanishing of $H_\mm^1(R/I)$.

\begin{Proposition} \label{h1}
Let $I$ be a monomial ideal in $R$. Then $H_\mm^1(R/I)  = 0$ if and only if   $\D_\a(I)$ is connected for all $\a \in \NN^n$ and
$\depth R_j/I_j  > 0$ for all $j = 1,...,n$, where $R_j = k[x_i| i \neq j]$ and $I_j = IR[x_j^{-1}]\cap R_j$.
\end{Proposition}

\begin{proof}
The assertion follows from a more precise result, namely, that for all $\a \in \ZZ^n$, $H_\mm^1(R/I)_\a  = 0$ if and only if one of the following conditions is satisfied:\par
{\rm (a)}  $\a \in \NN^n$ and $\D_\a(I)$ is connected,\par
{\rm (b)} $\a$ has only a negative component, say $a_j$ with $\depth R_j/I_j  > 0$,\par
{\rm (c)} $\a$ has more than one negative component. \par
\noindent To prove this result we use the formula $H_\mm^1(R/I)_\a  = \widetilde H_{-|G_\a|}(\D_\a(I),k)$ of  Theorem \ref{Takayama}.
If $|G_\a| = 0$, i.e.~$\a \in \NN^n$,  $H_\mm^1(R/I)_\a  = \widetilde H_0(\D_\a(I),k)$, which vanishes if and only if $\D_\a(I)$ is connected. If $|G_\a| = 1$, i.e.~$\a$ has only a negative component, then $H_\mm^1(R/I)_\a  = \widetilde H_{-1}(\D_\a(I),k)$, which vanishes if and only if $\D_\a(I) \neq \{\emptyset\}$. If this negative component is $a_j$, this condition is satisfied if and only if $\depth R_j/I_j > 0$ by Lemma \ref{empty}.  If $|G_\a| \ge 2$, i.e.~$\a$ has more than one negative component,  we have $H_\mm^1(R/I)_\a = H_{-t}(\D_\a(I),k)$ for some $t \ge 2$, which always vanishes.
\end{proof} 

It is difficult to give a meaningful criterion for the vanishing of $H_\mm^i(R/I)$, $i \ge 2$, because
it will involve the vanishing of $\widetilde H_i(\D_\a(I),k)$ for $i \ge 1$, that cannot be 
characterized by purely combinatorial means.

\section{Positive depth}

In this section we study the edge ideal of a hypergraph. Recall that a hypergraph is a system of subsets of a set.  
The given set is called the {\em vertex set} and the subsets the {\em edges} of the hypergraph (see e.g. \cite{Du}).
Let $\H$ be a hypergraph on the vertex set $[n] =\{1,...,n\}$. 
One can associate with $\H$ the {\em edge ideal} $I(\H)$ which is generated by the monomials $x^F$, $F \in \H$. 
The ideal $I(\H)$ is a squarefree monomial ideal. Note that we are not restricted to clutters which are hypergraphs with no containments among their edges though every square monomial ideal is the edge ideal of a clutter. \par

Let $\H$ be a hypergraph on the vertex set $V = [n]$. Let $I$ be the edge ideal of $\H$.
We will describe first the monomials of $\widetilde{I^2} \setminus I^2$ 
and then give a combinatorial criterion for $\depth R/I^2 > 0$.
For that we need the following notions. \par

We call a set $U \subseteq V$ {\em decomposable} in $\H$ if $U$ can be partitioned into two subsets each of them contains an edge of $\H$. Otherwise we call $U$ {\em indecomposable} in $\H$. Note that every set $U \subseteq V$ is decomposable if $\H$ contains the empty set. If $\H$ does not contain the empty set,  $U$ is decomposable if and only if $U$ contains two disjoint edges of $\H$.  
We call  $U$ a {\em 2-saturating set} of $\H$ if $U$ is indecomposable in $\H$ and $U \setminus i$ is decomposable in $\H(i)$ for every vertex $i \in V$, where $U \setminus i := U \setminus \{i\}$ and $\H(i) :=  \{F \setminus i|\ F \in \H\}$.   

\begin{Lemma} \label{saturation}
Let $\a \in \NN^n$ and $U =  \{i|\ a_i \neq 0\}$. 
Then $x^\a \in \widetilde{I^2} \setminus I^2$ if and only if $\a \in \{0,1\}^n$ and $U$ is a 2-saturating set of $\H$.
\end{Lemma}

\begin{proof}	
We have $\rho_j(I^2) = 2$ for all $j \in [n]$. 
Therefore, $H_\mm^0(R/I^2)_\a = 0$ for $\a \not\in \{0,1\}^n$ by Theorem \ref{Takayama}.
 Since $H_\mm^0(R/I^2) = \widetilde {I^2}/I^2$, 
$x^\a \in \widetilde{I^2} \setminus I^2$ if and only if $H_\mm^0(R/I^2)_\a \neq 0$. 
So we may assume that $\a \in \{0,1\}^n$.  \par

Under this assumption, $x^\a \not\in I^2$ if and only if $U$ is indecomposable in $\H$.
By the definition of the saturation, $x^\a \in \widetilde{I^2}$ if and only if 
for every $i \in [n]$, there exists $r \ge 0$ such that $x_i^rx^\a \in I^2$ 
or, equivalently, $x^{U \setminus i} \in I^2R[x_i^{-1}] \cap S$, 
where $S = k[x_j|\ j \neq i]$ and $x^{U \setminus i} = \prod_{j \in U \setminus i}x_j$. 
Let $J \subseteq S$ be the edge ideal of the hypergraph $\H(i)$.
Then $I^2R[x_i^{-1}] \cap S = J^2$. 
Therefore, $x^\a \in \widetilde{I^2}$ if and only if $U \setminus i$ is decomposable in $\H(i)$. 
\end{proof}

 \begin{Theorem} \label{positive depth}
Let $I$ be the edge ideal of a hypergraph $\H$. 
Then $\depth R/I^2 > 0$ if and only if $\H$ does not have any 2-saturating set.
\end{Theorem}

\begin{proof}	
By Lemma \ref{saturation}, every 2-saturating set of $\H$ corresponds to a monomial of $\widetilde {I^2}\setminus I^2$.
Since $\depth R/I^2 > 0$ means $H_\mm^0(R/I^2) = 0$, and $H_\mm^0(R/I^2) = \widetilde {I^2}/I^2$, the assertion follows.
\end{proof}

A 2-saturating set of $\H$ can be characterized in a more precise way by using terminologies of hypergraph theory. \par

For every subset $U \subseteq V$ one calls the subhypergraph $\H|_U := \{F \in \H|\ F \subseteq U\}$ the {\em section} (or {\em trace}) of $\H$ on $U$. A hypergraph is called {\em intersecting} if every pair of edges intersect. Note that a hypergraph is non-intersecting if it contains the empty set and that $U$ is indecomposable in $\H$ if and only if $\H|_U$ is intersecting. We call a hypergraph $\H$ {\it loosely intersecting} if it is intersecting but for every vertex $i$, there are two edges which intersects only at $i$. So the intersecting property is lost when we take out any vertex. More precisely, this means that the link of every vertex is not intersecting, where for a vertex $i$, the {\em link} of $i$ is the set $\lk_\H i :=  \{F \setminus  i|\ F \in \H, i \in F\}$. In the following we call a vertex {\em isolated} if this vertex is an edge of the hypergraph.

\begin{Lemma} \label{saturating}
$U$ is a 2-saturating set of $\H$ if and only if the following conditions are satisfied:\par
{\rm (a)} $\H|_U$ is loosely intersecting,\par
{\rm (b)} For every non-isolated vertex $i \not\in U$, $\lk_\H i|_U$ has two disjoint edges 
or an edge which is disjoint to an edge of $\H|_U$.  
\end{Lemma}

\begin{proof}
We may assume that $\H|_U$ is intersecting. Under this assumption we only need to show the following statements:  \par
(a') $U$ is decomposable in $\H(i)$ for all $i \in U$ if and only if (a) is satisfied, \par
(b') $U$ is decomposable in $\H(i)$ for all $i \not\in U$ if and only if  (b) is satisfied. \par

First  we note that $\H(i) = \H \cup \lk_\H i$.
Since $U$ is indecomposable in $\H$, $U$ is decomposable in $\H(i)$ if and only if
$U$ is decomposable in $\lk_\H i$ or $U$ contains an edge of $\lk_\H i$ and an edge of $\H$ which are disjoint.
If $i$ is an isolated vertex of $\H$, $U$ is decomposable in $\lk_\H i$ because $\lk_\H i$ contains the empty set. 
These facts immediately imply (b'). Concerning (a') it suffices to show that $U$ is decomposable in $\lk_\H i$ for $i \in U$ 
if  $U$ contains an edge of $\lk_\H i$ and an edge of $\H$ which are disjoint. \par
 
Let $F \subseteq U$ be an edge of $\lk_\H i$ and $G \subset U$ an edge of $\H$ such that $F \cap G = \emptyset$. 
If $i \not\in G$, this implies $(F \cup  i ) \cap G =  \emptyset$. 
Since $F \cup  i \subseteq U$ is an edge of $\H$, this contradicts the assumption that $\H|_U$ is intersecting. 
Therefore, $i \in G$. Hence $G \setminus  i \in \lk_\H i$. Since $F \cap (G \setminus  i ) = \emptyset$, 
$U$ is decomposable in $\lk_\H i$, as required.
\end{proof} 

Comparing with the definition of a 2-saturating set, the above characterization has the advantage that 
condition (a) concerns only the section of $\H$ on $U$, which is easier to check, 
while condition (b) reflects the complicated interplay between this section and the vertices outside of $U$.

\begin{Remark} \label{small}
{\rm A 2-saturating set can be empty or consist of only one element.
By Lemma \ref{saturation}, that does happen if and only if $I$ is the maximal homogeneous ideal, 
which means that every vertex of $\H$ is isolated.  
If $I$ is not the maximal ideal, a 2-saturating set has at least three elements by the following property of loosely intersecting hypergraphs.}
\end{Remark}

One calls a sequence of three edges $F_1, F_2, F_3$ a {\em special triangle} 
if there are vertices $v_1 \in F_1 \cap F_2 \setminus F_3$, $v_2 \in F_2 \cap F_3 \setminus F_1$, $v_3 \in F_1 \cap F_3 \setminus F_2$. 
We say that  such a special triangle has {\em empty intersection} if $F_1 \cap F_2 \cap F_3 = \emptyset$.  

\begin{Lemma} \label{special}
Every loosely intersecting hypergraph that has more than one vertex contains a special triangle with empty intersection.
\end{Lemma}

\begin{proof}	
Let $\G$ be a loosely intersecting hypergraph on a vertex set of more than one elements.
If $\G$ has an isolated vertex $v$, then every edge of $\G$ contains $v$ because $\G$ is intersecting.
Every vertex $w \neq v$ can not be an edge of $\H$ because it does not contain $v$.
Therefore, $\lk_\H w$ does not contain the empty set. Since $\lk_\H w$ is non-intersecting, 
$\lk_\H w$ has two disjoint edges, which contradicts the fact that they both contain $v$.
Thus, $\H$ has no isolated vertex. As we have just seen, this implies that the link of every vertex has two disjoint edges.
\par

Let $v_1$ be a vertex of $\G$. Since $\lk_\H i$ has two disjoint edges,  
$\G$ has two edges $F_1, F_2$ properly containing $v_1$ 
such that $F_1 \setminus v_1, F_2 \setminus v_1$ are disjoint. 
Choose $v_2  \in F_2$, $v_2 \neq v_1$. Then $\G$ also has two edges $G_1, G_2$ containing $v_2$ 
such that $G_1 \setminus v_2, G_2 \setminus v_2$ are disjoint. 
Therefore, both $G_1, G_2$ can not contain $v_1$. Let $G_1$ be the edge not containing $v_1$.
Since $\G$ is intersecting, there is a vertex $v_3 \in F_1 \cap G_1$. 
Clearly, $v _2 \not\in F_2$, $v_2 \not\in F_1$, $F_1 \cap F_2 \cap F_3 = \emptyset$. 
Hence $F_1, F_2, F_3$ form a special triangle with empty intersection. 
\end{proof}	

A special triangle in a graph is just a triangle, which always has an empty intersection.
In the following {\em we always assume that a graph is a collection of subsets of two elements of a vertex set $V$}.

\begin{Corollary} \label{triangle}
A graph is loosely intersecting if and only if it is a triangle.
\end{Corollary}

\begin{proof}	
Let $\G$ be an loosely intersecting graph. Then $\G$ has a triangle by Lemma \ref{special}.
Since every edge of $\G$ must intersect the three edges of the triangle, it must coincide with one of them.
This show that $\G$ is a triangle. Obviously, every triangle is loosely intersecting.
\end{proof} 

In graph theory, one calls a set $U$ of vertices {\em dominating} if every vertex is adjacent to at leat one vertex of $U$ (see e.g. \cite{Ha}). 
For this reason we say that a triangle of a graph is dominating if their vertices form a dominating set.

\begin{Lemma} \label{dominating}
A set of vertices of a graph is 2-saturating if and only if  they are the vertices of a dominating triangle.
\end{Lemma}

\begin{proof}	
Let $U$ be a 2-saturating set of a graph $\H$. By Lemma \ref{saturating}(a) and 
Corollary \ref{triangle}, $U$ is the vertex set of a triangle. 
By Lemma \ref{saturating}(b), every vertex is adjacent  to at leat one vertex of $U$. 
Hence $U$ is dominating. \par
Conversely, assume that $U$ is the set of vertices of a dominating triangle. 
This triangle is a loosely intersecting section of $\H$ by Corollary \ref{triangle}.
Since $U$ is a dominating set,  every vertex $i \not\in U$ is adjacent to at leat one vertex $j$ of $U$.
This vertex $j$ is an edge of $\lk_\H i$ and disjoint to the edge of the triangle not containing $j$. 
Therefore, $U$ is 2-saturating by Lemma \ref{saturating}.  
\end{proof}

Theorem \ref{positive depth} together with Lemma \ref{dominating} yield the following result.

\begin{Theorem} \label{depth > 0 graph}
Let $I$ be the edge ideal of a graph $\H$. 
Then $\depth R/I^2 > 0$ if and only if  $\H$ has no dominating triangle. 
\end{Theorem}

\begin{Example}
{\rm Let us consider the following graphs which contain a triangle:

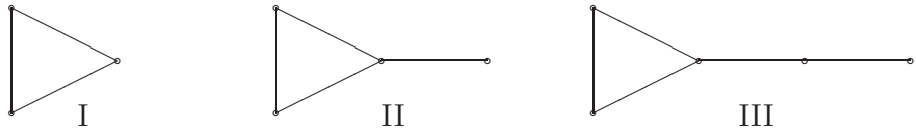
\begin{picture}(0,80)
\put(50,20){\line(2,1){40}}
\put(50,20){\line(0,1){40}}
\put(50,60){\line(2,-1){40}}
\put(50,20){\circle{2}}
\put(50,60){\circle{2}}
\put(90,40){\circle{2}}
\put(75,15){I}

\put(150,20){\line(2,1){40}}
\put(150,20){\line(0,1){40}}
\put(150,60){\line(2,-1){40}}
\put(190,40){\line(1,0){40}}
\put(150,20){\circle{2}}
\put(150,60){\circle{2}}
\put(190,40){\circle{2}}
\put(230,40){\circle{2}}
\put(190,15){II}

\put(270,20){\line(2,1){40}}
\put(270,20){\line(0,1){40}}
\put(270,60){\line(2,-1){40}}
\put(310,40){\line(1,0){40}}
\put(350,40){\line(1,0){40}}
\put(270,20){\circle{2}}
\put(270,60){\circle{2}}
\put(310,40){\circle{2}}
\put(350,40){\circle{2}}
\put(390,40){\circle{2}}
\put(325,15){III}
\end{picture}

\noindent By Theorem \ref{depth > 0 graph}, $\depth R/I^2 = 0$  for  I and II, whereas $\depth R/I^2 > 0$ for  III.}
\end{Example}

In general, we need to know which hypergraph is loosely intersecting in order 
to  find a combinatorial criterion for $\depth R/I^2 > 0$.
The following example shows that there doesn't exit a unique loosely intersecting $r$-uniform hypergraph for $r \ge 3$. 
Recall that a hypergraph is called {\em $r$-uniform} if every edge has $r$ vertices. 

\begin{Example}
{\rm The following 3-uniform hypergraphs are loosely intersecting:
\begin{align*}
& \{1,2,3\}, \{2,3,4\},\{3,4,5\},\{1,4,5\},\{1,2,5\} \text{\ or}\\
& \{1,2,3\}, \{1,3,5\}, \{1,4,5\},\{2,3,4\},\{2,4,5\}  \text{\ or}\\
& \{1,2,3\}, \{1,4,5\},\{2,3,4\},\{2,3,5\}, \{2,4,5\},\{3,4,5\}.
\end{align*}
These hypergraphs are minimal in the sense that they don't contain proper partial hypergraphs which are loosely intersecting.}
\end{Example}

\begin{Remark} 
{\rm One may raise the question whether the above method can be used to study $I^t$ for $t \ge 3$. 
The answer is yes but it is more complicated. We need to know when a monomial $x^\a \in \widetilde{I^t} \setminus I^t$. 
To check this condition we  consider the multiset $U_\a$ 
which consists of $a_i$ copies of $i$, $i = 1,...,n$. 
A multiset is said to be {\em $t$-decomposable} in $\H$ if it contains a union of $t$ edges of $\H$. 
We call $U_\a$ {\em $t$-saturating}  if  $U_\a$ is $t$-indecomposable in $\H$ and $U_\a \setminus \{a_i\ \text{copies of}\ i\}$ is $t$-decomposable in $\H(i)$ for all $i=1,...,n$. Similarly as above, we can prove that $x^\a \in \widetilde{I^t} \setminus I^t$ 
if and only if  $U_\a$ is $t$-saturating. }
\end{Remark}

\section{Associated primes}

In this section we will apply Theorem \ref{positive depth} 
to study the associated primes of the second power of the edge ideal of a hypergraph. \par

Let $\H$ be a hypergraph on a vertex set $V = [n]$. Let $C$ be a subset of $V$. 
One calls $C$ a (vertex) {\em cover} of $\H$ if $C$ meets every edge (see e.g \cite{HHT}).
Set $\H_C := \{F \cap C| \ F \in \H\}$. One calls $\H_C$ the {\em induced subhypergraph} of $\H$ on $C$. 
\par

Let $I$ be the edge ideal of $\H$. We denote by $P_C$ the ideal generated by the variables $x_i$, $i \in C$.
It is well known that every associated prime of $I^2$ has the form $P_C$ for some cover $C$ of $\H$. 
Such a cover can be characterized as follows.
  
\begin{Theorem} \label{asso}
Let $I$ be the edge ideal of a hypergraph $\H$. For a subset $C$ of the vertex set, 
$P_C$ is an associated prime of $I^2$ if and only if $\H_C$ has a 2-saturating set. 
\end{Theorem}

\begin{proof}	
Let $P = P_C$ and $A = R[x_i^{-1}|\ i \not\in C]$. 
Then $P$ is an associated prime of $I^2$ if and only if $PA$ is an associated prime of $I^2A$. 
Let $S = k[x_i|\ i \in C]$ and $Q$ the ideal of $S$ generated by the variables $x_i$, $i \in C$.
Let $J$ denote the edge ideal of $\H_C$ in $S$. Then $A = S[x_i^{\pm 1}|\ i \not\in C]$, $PA = QA$, and $I^2A = J^2A$.
Since $A$ is a flat extension of $S$, $PA$ is an associated prime of $I^2A$ 
if and only if $Q$ is an associated prime of $J^2$. The latter condition means $\depth S/J^2 = 0$. 
Hence the assertion follows from Theorem \ref{positive depth} (applied to the hypergraph $\H_C$).
\end{proof}

For brevity we say that $U$ is a 2-saturating (or loosely intersecting) set of $C$ if 
$U$ is a 2-saturating set of $\H_C$ (or $\H_C|_U$ is loosely intersecting).
Note that these conditions implies that $C$ is a cover of $\H$ because otherwise $V \setminus C$ contains an edge of $\H$, hence 
$\H_C$ contains the empty set, which contradicts the indecomposability of $U$ in $\H_C$.
In general, a cover having a 2-saturating set can be characterized as follows.

\begin{Proposition} \label{minimal}
A cover $C$ of $\H$  has a 2-saturating set $U$ if and only if 
$C$ is minimal among the covers $D$ of $\H$ having $U$ as a loosely intersecting set. 
\end{Proposition}

\begin{proof}	
Assume that $C$ has a 2-saturating set $U$.
Assume that $C$ is not minimal in the above sense.
Then there exists a cover $D \subset C$ of $\H$ containing $U$ such that $H_D|_U$ is loosely intersecting.
Let $i$ be a vertex in $C \setminus D$. Then $i$ is a non-isolated vertex.
By Lemma \ref{saturating}, $\lk_{\H_C}i|_U$ has two disjoint edges or an edge which doesn't intersect an edge of $H_C|_U$. 
Since $\lk_{\H_C}(i)|_U \subseteq H_D|_U$ and $\H_C|_U \subseteq \H_D|_U$,
this implies that $H_D|_U$ is non-intersecting, a contradiction.\par

Conversely, assume $C$ is minimal among the covers $D$ of $\H$ 
having $U$ as a loosely intersecting set.  
Let $i \in C \setminus U$ be an arbitrary non-isolated vertex of $\H_C$. 
Then $F \cap C  \neq \{i\}$ for every edge $F$ of $\H$.
Therefore, $C' = C \setminus \{i\}$ is a cover of $\H$.
By the minimal property of $C$, $\H_{C'}|_U$  is not loosely intersecting.
Note that $\H_{C'}|_U = H_C|_U \cup \lk_{H_C}i|_U$.
Since the link of every vertex in  $H_C|_U$ is non-intersecting, 
the link of every vertex in $H_{C'}|_U$ is also non-intersecting.
Therefore, $\H_{C'}|_U$ is not intersecting.
Since $H_C|_U$ is intersecting, this implies that $\lk_{H_C}i|_U$ has two disjoint edges or an edge 
which does not intersect an edge of $H_C|_U$.  
By Lemma \ref{saturating}, $U$ is a 2-saturating set of $\H_C$.
\end{proof}

Since a loosely intersecting set can be easily detected, 
we can use the above lemma to work out an algorithm  to find the associated primes of $I^2$. \par

Now we are going to describe properties of 2-saturating sets of a cover, 
which will be useful for the description of the associated primes of $I^2$. 
These properties depend on whether the cover is minimal or non-minimal, 
which corresponds to minimal or non-minimal (i.e. embedded) associated primes of $I^2$.

\begin{Lemma} \label{minimal cover}
Let $U$ be a subset of $V$. Then $U$ is a 2-saturating set of a minimal cover $C$ of $\H$ if and only if 
$U$ is the empty set or a vertex of $C$. 
Moreover, the minimal covers are the only covers that  have a 2-saturating set of less than two elements. 
\end{Lemma}

\begin{proof}  
For every vertex $i \in C$, there exists an edge $F$ such that $F \cap C = \{i\}$ because $C \setminus i$ is not a cover of $\H$. 
From this it follows that every vertex of $C$ is isolated in $\H_C$. Hence the assertions follows from 
Remark \ref{small}. 
\end{proof}

\begin{Lemma} \label{non-minimal}
Let $U$ be a 2-saturating set of a non-minimal cover $C$ of $\H$. 
Then $\H|_U$ has a special triangle $F_1,F_2,F_3$ such that 
$(F_1 \cap F_2 \cap F_3) \cap C = \emptyset$ and $(F_1 \cup F_2 \cup F_3)\cap C$ is indecomposable in $\H_C$.
\end{Lemma}

\begin{proof}	 
Since $C$ is not a minimal cover, $\H_C$ has at least a non-isolated vertex $i$. 
Hence $\H_C(i)$ doesn't contain the empty set.
Since $U$ is decomposable in $\H_C(i)$, $U$ contains at least two different edges of $\H_C(i)$. 
As a consequence, $U$ has at least two vertices. 
By Lemma \ref{special}, $\H_C|U$ contains a special triangle $F_1',F_2',F_3'$ with $F_1' \cap F_2' \cap F_3' = \emptyset$.
Let $F_i \in \H$ such that $F_i' = F_i \cap C$. 
Then $F_1,F_2,F_3$ is a special triangle of $H$ with $F_ 1 \cap F_2 \cap F_3 \cap C = \emptyset$. 
Since $U$ is indecomposable in $\H_C$ and since $(F_1 \cup F_2 \cup F_3) \cap C \subseteq U$, $(F_1 \cup F_2 \cup F_3) \cap C$ is indecomposable in $\H_C$. 
\end{proof}

Let $I^{(2)}$ denote the second symbolic power of $I$, 
that is the intersection of the primary components of the minimal associated primes of $I^2$. 
Then $I^{(2)} = I^2$ if and only if $I^2$ has no non-minimal associated primes.
Since these primes correspond to non-minimal covers of $\H$ having 2-saturating sets,  
we can use Lemma \ref{non-minimal} to deduce a criterion for $I^{(2)} = I^2$.
 
\begin{Theorem} \label{symbolic}
Let $I$ be the edge ideal of a hypergraph $\H$. Then $I^{(2)} = I^2$ if and only if $\H$ doesn't contain any
special triangle $F_1,F_2,F_3$ such that  $(F_1 \cup F_2 \cup F_3) \cap D$ is indecomposable in $\H_D$
for $D = V \setminus (F_1 \cap F_2 \cap F_3)$.
\end{Theorem}

\begin{proof}	
Assume that $I^{(2)} \neq I^2$. Then there is a non-minimal cover $C$ of $\H$ having a 2-saturating set $U$.
By Lemma \ref{non-minimal}, $\H$ has a special triangle $F_1,F_2,F_3$ such that 
$F_ 1 \cap F_2 \cap F_3 \cap C = \emptyset$ and  $(F_1 \cup F_2 \cup F_3) \cap C$ is indecomposable in $\H_C$.
Let $D = V \setminus (F_ 1 \cap F_2 \cap F_3)$. Then $C \subseteq D$. 
Since every edge of $\H_C$ is the intersection of an edge of $\H_D$ with $C$, 
$(F_1 \cup F_2 \cup F_3) \cap D$ is also indecomposable in $\H_D$. \par

Conversely, assume that $\H$ has a special triangle $F_1,F_2,F_3$ such that 
$(F_1 \cup F_2 \cup F_3) \cap D$ is indecomposable in $\H_D$.
Let $F_i' = F_i \cap D$. Then $F_1',F_2',F_3'$ is a special triangle in $\H_D$ with empty intersection. 
Thus, there exists a smallest set $C$ such that $\H_C$ contains a special triangle $G_1,G_2,G_3$ with empty intersection
and $U := G_1 \cup G_2 \cup G_3$ is indecomposable in $\H_C$. 
We will show that $U$ is decomposable in $\H_C(i)$ for all $i \in C$.
For that we may assume that $\H_C(i)$ does not contain the empty set, i.e. 
$i$ is not an isolated vertex of $\H_C$. Then $F \cap C \neq \{i\}$ for all $F \in \H$.
Put $C' := C \setminus i$. Then $\H_C(i) = \H_{C'}$.
If $U$ is indecomposable in $\H_{C'}$,
then $G_j \cap G_h \neq \{i\}$ for $j, h = 1,2,3$. Put $G_j' = G_j \setminus i$. 
Then $G_j' \cap G_h' \neq \emptyset$.  Hence $G_1',G_2',G_3'$ is a special triangle in $\H_{C'}$ with empty intersection.
Since $U$ is indecomposable in $\H_{C'}$, $G_1' \cup G_2' \cup G_3' = U \setminus i$ is indecomposable in $\H_{C'}$.
So we obtain with $C'$ a contradiction to the minimal property of $C$.
This shows  $U$ is a 2-saturating set of $C$. Since $|U| \ge 3$,  $C$ is not a minimal cover of $\H$ by Remark \ref{minimal cover}. 
Therefore, $P_C$ is not a minimal prime over $I^2$. This implies $I^{(2)} \neq I^2$.
\end{proof}

Theorem \ref{symbolic} can be also deduced from the following result of Rinaldo, Terai and Yoshida, 
which was proved by a different method.

\begin{Corollary}\label{symbolic} \cite[Theorem 2.1]{RTY2}
$I^{(2)} = I^2$ if and only if $\H$ has no special triangle $F_1,F_2,F_3$ such that
$x^{F_1 \cup F_2 \cup F_3}x^{F_1 \cap F_2 \cap F_3} \not\in I^2$.
\end{Corollary}

\begin{proof}	
Assume that  $\H$ has a special triangle $F_1,F_2,F_3$ such that 
$(F_1 \cup F_2 \cup F_3)\cap D$ is indecomposable in $\H_D$ for $D = V \setminus (F_1 \cap F_2 \cap F_3)$.
Then every pair of edges of $\H_D$ in $F_1 \cup F_2 \cup F_3$ intersect.
Therefore, every pair of edges of $\H$ in $F_1 \cup F_2 \cup F_3$ must have at least a common vertex in $D$. 
This is equivalent to the condition $x^{F_1 \cup F_2 \cup F_3}x^{F_1 \cap F_2 \cap F_3} \not\in I^2$.

Conversely, if $\H$ has a special triangle $F_1,F_2,F_3$ such that
$x^{F_1 \cup F_2 \cup F_3}x^{F_1 \cap F_2 \cap F_3} \not\in I^2$,
then there doesn't exist any pair of edges in $F_1 \cup F_2 \cup F_3$ whose intersection is contained in $F_1 \cap F_2 \cap F_3$.
Therefore, every pair of edges in $F_1 \cup F_2 \cup F_3$ have at least a common vertex in $D$.
Hence $(F_1 \cup F_2 \cup F_3) \cap D$ is indecomposable in $\H_D$. 
\end{proof}

The above results become simpler in the graph case. 
For a set $U \subseteq V$ we denote by $N(U)$ the set of all vertices which are adjacent to vertices of $U$. 
One calls $N(U)$ the {\em neighborhood} of $U$. We say that $U$ is a triangle of a graph if $U$ is the vertex set of a triangle.

\begin{Lemma} \label{neighbor}
Let $\H$ be a graph. A set $U \subseteq V$ is a 2-saturating set of a non-minimal cover $C$ of $\H$ 
if and only if $U$ is  a triangle and $C$ is minimal among the covers of $\H$ containing $N(U)$.
\end{Lemma}

\begin{proof}	
Since $C$ is not minimal, $\H_C$ has at least a non-isolated vertex. Let $C'$ denote the set of non-isolated vertices of $\H_C$.
Then $\H_{C'}$ is a graph. If $U$ is a 2-saturating set of $C$, then $U$ contains a triangle by Lemma \ref{non-minimal}.
Hence $U$ does not contain any isolated vertex of $\H_C$. This implies $U \subseteq C'$. It is clear that 
$U$ is a  loosely intersecting set in $C'$. By Lemma \ref{triangle}, $U$ is  a triangle of $\H$. \par 
To prove the assertion we may now assume that $U$ is a triangle.
For any set $D \supseteq U$, $U$ is  loosely intersecting in $D$ if and only if  $N(U) \subseteq D$.   
By Proposition \ref{minimal}, this implies that $U$ is a 2-saturating set of $C$ 
if and only if $C$ is minimal among the covers $D$ containing $N(U)$.  
 \end{proof}

\begin{Theorem} \label{graph asso}
Let $I$ be the edge ideal of a graph $\H$. For a cover $C$ of $\H$, $P_C$ is an associated prime of $I^2$ if and only if  
$C$ is a minimal cover or $C$ is minimal among the covers containing the neighborhood of a triangle.
\end{Theorem}

\begin{proof}	 
Since $P_C$ is a minimal prime of $I^2$ if and only if $C$ is a minimal cover, we only need to prove that
$P_C$ is an embedded associated prime of $I^2$ if and only if  
$C$ is minimal among the covers containing the neighborhood of a triangle.
Therefore, the assertion follows from Theorem \ref{asso} and Lemma \ref{neighbor}.
 \end{proof}
  
By the above theorem, every embedded associated prime of $I^2$ originates from a triangle of the graph. 
As an immediate consequence, $I^{(2)} = I^2$ if and only if $\H$ has no triangle.
This result is only a special case of a more general result which says that
$I^{(t)} = I^t$ if and only if $\H$ has no odd cycle of length $\le 2t-1$ \cite[Lemma 3.10]{RTY1}. \par

The following example shows that $I^2$ may have different embedded associated primes which originate from the same triangle. 

\begin{Example}
{\rm Let us consider the following graphs which contain only a triangle:

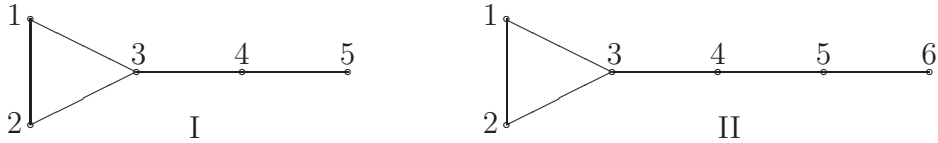
\begin{picture}(0,80)
\put(50,20){\line(2,1){40}}
\put(50,20){\line(0,1){40}}
\put(50,60){\line(2,-1){40}}
\put(90,40){\line(1,0){40}}
\put(130,40){\line(1,0){40}}
\put(50,20){\circle{2}}
\put(50,60){\circle{2}}
\put(90,40){\circle{2}}
\put(130,40){\circle{2}}
\put(170,40){\circle{2}}
\put(110,15){I}
\put(41,58){1}
\put(41,17){2}
\put(88,43){3}
\put(127,43){4}
\put(167,43){5}

\put(230,20){\line(2,1){40}}
\put(230,20){\line(0,1){40}}
\put(230,60){\line(2,-1){40}}
\put(270,40){\line(1,0){40}}
\put(310,40){\line(1,0){40}}
\put(350,40){\line(1,0){40}}
\put(230,20){\circle{2}}
\put(230,60){\circle{2}}
\put(270,40){\circle{2}}
\put(310,40){\circle{2}}
\put(350,40){\circle{2}}
\put(390,40){\circle{2}}
\put(310,15){II}
\put(221,58){1}
\put(221,17){2}
\put(268,43){3}
\put(307,43){4}
\put(347,43){5}
\put(387,43){6}
\end{picture}

\noindent Applying Theorem \ref{graph asso} we can see that for I, 
$I^2$  has only an embedded associated prime $(x_1,x_2,x_3,x_4)$,  
and for II,  $I^2$ has two embedded associated primes 
$(x_1,x_2,x_3,x_4,x_5)$ and $(x_1,x_2,x_3,x_4,x_6)$.}
\end{Example}

\section{Depth greater than one}

Let $\H$ be a hypergraph on a vertex set $V = [n]$. Let $I$ be the edge ideal of $\H$. 
In this section we want to study when $\depth R/I^2 > 1$.
It is known that  this is equivalent to the condition $H_\mm^i(R/I^2) = 0$ for $i = 0,1$. \par

In Section 2 we have given a combinatorial criterion for $\depth R/I^2 > 0$ or, equivalently, for $H_\mm^0(R/I^2) = 0$.
It remains to find a combinatorial criterion for $H_\mm^1(R/I^2) = 0$. By Proposition \ref{h1} we need to know when 
$\D_\a(I^2)$ is connected for all $\a \in \NN^n$.
\par

Set $H_\a = \{i|\ a_i > 0\}$. For $F \subseteq V$ we denote by $\overline F$ the complement of $F$ in $V$.
 
\begin{Lemma} \label{facet 2}
$\F(\D_\a(I^2)) = \{F|\  H_\a \cap \overline F \text{ is a 2-saturating set of } \overline F\}.$ 
\end{Lemma}

\begin{proof} 
By Lemma \ref{facet} we have
$\F(\D_\a(I^2))  := \{F\mid  x^{\a_F} \in \widetilde {I^2_F} \setminus I^2_F\},$
where $I_F = IR_F \cap k[x_i|\ i \in \overline F]$. 
It is easy to check that $I_F$ is the edge ideal of the hypergraph $\H_{\overline F}$ and $H_{\a_F} = H_\a \cap \overline F$.
By Lemma \ref{saturation}, $x^{\a_F} \in \widetilde {I^2_F} \setminus I^2_F$ if and only if 
$H_\a \cap \overline F$ is a 2-saturating set of $\overline F$.        
\end{proof}

Let $\D$ denote the simplicial complex $\D(I^2)$.
Since $\D = \D(I)$, $\D$ is the complex of the independent sets of $\H$, 
where  a subset $F \subseteq V$ is called {\it independent} if $F$ doesn't contain any edge of $\H$.
Obviously, $F$ is independent if and only if $\overline F$ is a cover of $\H$. 
For every vertex $i$ set $\st_\D i = \{F \in \D|\ i \in F\} \cup\{\emptyset\}$. One calls $\st_\D i$ the {\em star} of $i$ in $\D$.

\begin{Corollary}\label{star}
Let $\a = \e_i + \e_j$, $i \neq j$, where $\e_i$ and $\e_j$ denote the $i$-th and $j$-th unit vectors. 
Then  $\D_\a(I^2) = \st_\D i \cup \st_\D j.$  
\end{Corollary}

\begin{proof} 
We have $H_\a = \{i,j\}$.  
By Lemma \ref{facet 2}, the facets of $\D_a$ are the independent sets $F$ of $\H$ 
such that $\{i,j\} \cap \overline F$ is a 2-saturating set of $\overline F$.
By Remark \ref{small}, a 2-saturating set of a cover of $\H$ can be empty or have one vertex or more than 2 vertices.
Therefore, we either have $\{i,j\} \cap \overline F \subseteq \{i\}$ or $\{i,j\} \cap \overline F  \subseteq \{j\}$.
Moreover, if $\overline F$ has a 2-saturating set of fewer than two vertices, then $\overline F$ must be a minimal cover by Lemma \ref{minimal cover}. Therefore $F$ must be a facet of $\D$ containing $i$ or $j$.                     
\end{proof}

In graph theory one defines the {\em diameter} of a graph $\G$ as the maximal distance between two vertices and denotes it by $\diam \G$,
where the distance is the minimal length of paths between the vertices. Let $\D^{(1)}$ denote the graph of the 1-dimensional faces of $\D$, the {\em 1-dimensional skeleton} of $\D$.

\begin{Theorem} \label{depth > 1}
Let $I$ be the edge ideal of a hypergraph $\H$. Assume that  $\depth R/I^2 > 1$. Then $\diam \D^{(1)} \le 2$. 
\end{Theorem}
 
\begin{proof}	
It is known that $\depth R/I^2 > 1$ implies $H_\mm^1(R/I^2) = 0$.
By Proposition \ref{h1}, this implies that $\D_\a(I^2)$ is connected for all $\a \in \NN^n$.
Hence $\st_\D i \cup \st_\D j$ is connected for all $i \neq j$ by Corollary \ref{star}.
This means that the distance between $i,j$ is not greater than $2$. Thus, $\diam \D^{(1)} \le 2$. 
\end{proof}

It is sometimes better to formulate the above condition on $\D$ in terms of the hypergraph $\H$,

\begin{Lemma}\label{non-edge}
$\diam \D^{(1)} \le 2$ if and only if for every edge $\{i,j\}$ of $\H$, 
there exists a non-isolated vertex $h \neq i, j$ such that $\{i,h\}$ and $\{j,h\}$ are non-edges of $\H$.
\end{Lemma}

\begin{proof}
We note first that every vertex of $\D$ is a non-isolated vertex of $\H$ and 
that $\{i,j\} \in \D^{(1)}$ if and only if $\{i,j\}$ is a non-edge of $\H$.
Therefore, $\dist_\D(i,j) \le 2$ if and only if $\{i,j\}$ is a non-edge or 
there exists a non-isolated vertex $h \neq i, j$ such that $\{i,h\}$ and $\{j,h\}$ are non-edges of $\H$.
This implies the assertion.
\end{proof}

Rinaldo, Terai and Yoshida  showed that $\depth R/I^{(2)} > 1$ if and only if $\diam \D^{(1)} \le 2$  \cite[Theorem 3.2]{RTY2}. 
There are examples such that $\diam \D^{(1)} \le 2$ but $\depth R/I^2 \le 1$. 

\begin{Example}
{\rm Let $\H$ be the hypergraph of all 3-elements subsets of $\{1,2,3,4\}$.
Then $\D$ is the complex of a complete graph on 4 vertices. Hence $\diam \D^{(1)} = 1$.
Since $\{1,2,3,4\}$ is a 2-saturating set of $\H$,  $\depth R/I^2 = 0$ by Theorem \ref{positive depth}.}
\end{Example}

We couldn't find a general criterion for $\depth R/I^2 > 1$.
The reason is that we don't know when $\D_\a(I^2)$ is connected for  $\a \in \NN^n$ with
$|H_\a| \ge 3$ except in the graph case, where we have the following description.

\begin{Lemma} \label{disconnect}
Let $\H$ be a graph and  $\a \in \NN^n$ such that $|H_\a| \ge 3$.
If $\D_\a(I^2)$ is disconnected, then $H_\a$ is a triangle of $\H$.
\end{Lemma}

\begin{proof}	
Assume that $\D_\a(I^2)$ is disconnected.
Then $\D_\a(I^2)$ has two disjoint facets, say $F, G$. 
For brevity set $F' = H_\a \cap \overline F$ and  $G' = H_\a \cap \overline G$. 
By Lemma \ref{facet 2}, $F'$ and  $G'$ are 2-saturating sets of $\overline F$ and $\overline G$, respectively.
Since $V = \overline F \cup \overline G$, $H_\a = F'  \cup G'$.
Hence, one of these two sets, say $F'$ must have at least 2 elements.
By Remark \ref{small}, we must have $|F'| \ge 3$. Therefore, 
$F'$ is  a triangle and $\overline F$ contains $N(F')$ by Lemma \ref{dominating}. \par

If $|G'| \le 1$, then $|H_\a \setminus G'| \ge 2$.
Since $H_\a \setminus G' \subseteq F'$, $H_\a \setminus G'$ contains an edge of the triangle.
This contradicts the fact that $H_\a \setminus G' \subseteq G$ is an independent set.
Therefore, $|G'| \ge 3$ by Remark \ref{small}. Similarly as above, this implies that $G'$ is a triangle.
If $|G' \setminus F'| \ge 2$, then $G' \setminus F'$ contains an edge of this triangle.
This contradicts the fact that $G' \setminus F' \subseteq H_\a \setminus \overline F \subseteq F$ is an independent set.
If $|G' \setminus F'| = 1$, then $|G' \cap F'| = 2$. 
Hence the vertex of $G' \setminus F'$ is adjacent to at least two vertices of $F'$. 
As a consequence, this vertex belongs to $N(F') \cap F$.
This contradicts the fact that $\overline F$ contains $N(F')$.
So we have shown that $|G' \setminus F'| = 0$, which implies $G' = F'$ because $|G'| = |F'| = 3$.
Thus, $H_\a = G' = F'$ is a triangle.
\end{proof}

Let $\D_U$ denote the induced subcomplex of $\D$ on a subset $U \subseteq V$.

\begin{Lemma} \label{induced}
Let $\H$ be a graph and $\a \in \NN^n$ such that $H_\a$ is a triangle. Then
$\D_\a(I^2) = \D_{\overline{N(H_\a)}}$, where $\overline{N(H_\a)}$ denotes the complement of $N(H_\a)$ in $V$.
\end{Lemma}

\begin{proof}	
Let $F$ be a facet of $\D_\a(I^2)$. 
By Lemma \ref{facet 2},  $H_\a \cap \overline F$ is a 2-saturating set of  $\overline F.$
Since $\overline F$ is a cover of $\H$, $H_\a \setminus \overline F$ is an independent set.
Hence $H_\a \setminus \overline F$ doesn't contain any edge of the triangle.
Therefore $|H_\a \setminus \overline F| \le 1$, which implies  $|H_\a \cap \overline F| \ge 2$.
By Remark \ref{small},  we must have $|H_\a \cap \overline F| = 3 = |H_\a|$.   
Hence $H_\a =  H_\a \cap \overline F$.
By Lemma \ref{neighbor}, $\overline F$ is minimal among the covers of $\H$ which contains $N(H_\a)$.
This means $F$ is maximal among the independent sets in $\overline{N(H_\a)}$.
Hence $F$ is a facet of $\D_{\overline{N(H_\a)}}$. \par
For the converse let $F$ be a facet of $\D_{\overline{N(H_\a)}}$. 
Then $\overline F$ is minimal among the covers of $\H$ which contain $N(H_\a)$.
By Lemma \ref{neighbor}, $H_\a$ is a 2-saturating set of $\overline F$.
Hence $F$ is a facet of $\D_\a(I^2)$ by Lemma \ref{facet 2}. 
So we have proved that $\F(\D_\a(I^2)) = \F(\D_{\overline{N(H_\a)}})$, which implies $\D_\a(I^2) = \D_{\overline {N(H_\a)}}$. 
\end{proof}	

Using the above results we can give a combinatorial criterion for $\depth R/I^2 > 1$.

\begin{Theorem} \label{depth > 1 graph}
Let $I$ be the edge ideal of a graph $\H$.  Let $\overline \H$ denote the graph of the non-edges of $\H$.
Then $\depth R/I^2 > 1$ if and only if  the following conditions are satisfied:\par
{\rm (a)} $\diam \overline \H \le 2$,\par
{\rm (b)} For every triangle $U$ of $\H$, $\overline{N(U)}$ has at least two elements and 
the induced subgraph $\overline \H|_{\overline{N(U)}}$ is connected.
\end{Theorem}

\begin{proof}	
We know that $\depth R/I^2 > 1$ if and only if $H_\mm^i(R/I^2) = 0$ for $i = 0,1$.
By Theorem \ref{depth > 0 graph}, $H_\mm^0(R/I^2) = 0$ if and only if $\H$ has no dominating triangle.
This means that $\overline{N(U)} \neq \emptyset$ for any triangle $U$ of $\H$.
Under this condition, we only need to prove that $H_\mm^1(R/I^2) = 0$ if and only if (a) and (b) are satisfied. 
By Proposition \ref{h1}, $H_\mm^1(R/I^2) = 0$  if and only if 
$\D_\a(I^2)$ is connected  for all $\a \in \NN^n$ and $\depth R_j/I_j> 0$ for all $j \in V$, 
where $R_j = k[x_i| i \neq j]$ and $I_j = IR[x_j^{-1}]\cap R_j$. 
\par

Assume that $H_\mm^1(R/I^2) = 0$.
Then (a) is satisfied by Proposition \ref{depth > 1} because $\D^{(1)} = \overline \H$.
To prove (b) let $U$ be an arbitrary triangle of $\H$ and $j \in \overline{N(U)}$. 
Let  $C = V \setminus j$. Then $I_j$ is the edge ideal of the induced subhypergraph $\H_C$. 
Since $\depth R_j/I_j > 0$, $\H_C$ has no dominating triangle by Theorem \ref{depth > 0 graph}. 
Hence $N(U) \neq C$. Therefore,  $\overline{N(U)}$ has at least two elements.
Let $\a \in \NN^n$ such that $H_\a = U$. 
By Lemma \ref{induced}, $\D_\a(I^2) = \D_{\overline{N(U)}}$.
Hence $\D_{\overline{N(U)}}$ is connected. Since $\overline \H$ is the one-dimensional skeleton of $\D$,
$\overline \H|_{\overline{N(U)}}$  is the one-dimensional skeleton of $\D_{\overline{N(U)}}$.
Therefore, $\overline \H|_{\overline{N(U)}}$ is also connected. \par

Conversely, assume that (a) and (b) are satisfied. 
From the condition that $\overline{N(U)}$ has at least two elements for any triangle $U$ of $\H$ 
we can show similarly as above that $\depth R_j/I_j > 0$ for all $j \in V$.
It remains to show that  $\D_\a(I^2)$ is connected  for all $\a \in \NN^n$.
By Theorem \ref{Takayama} we may assume that $\a \in \{0,1\}^n$.
If $|H_\a| \le 1$,  then $\a = \0$ or $\a = \e_i$ for some $i$. By Remark \ref{initial}, $\D_\a(I^2) = \D$, which is 
connected by (a). If  $|H_\a| = 2$, then $\a = \e_i + \e_j$ for some $i \neq j$. By Corollary \ref{star}, $\D_\a(I^2) = \st_\D i \cup \st_\D j$, which is connected by (a).
If $|H_\a| \ge 3$, we may assume that $H_\a$ is a triangle by Lemma \ref{disconnect}.
Then $\D_\a(I^2) = \D_{\overline{N(H_\a)}}$ by Lemma \ref{induced}.  Hence $\D_\a(I^2)$ is connected by (b).
\end{proof}

\begin{Corollary} \label{bipartite}
Let $I$ be the edge ideal of a bipartite graph $\H$.  Let $F$ and $G$ be a partition of the vertices of $\H$ such that every edge connects a vertex in $F$ to a vertex in $G$. Then $\depth R/I^2 > 1$ if and only if there exist neither a vertex $i \in F$ with $N(i) = G$ nor a vertex $j \in G$ with $N(j) = F$.
\end{Corollary}

\begin{proof} 
Since $\H$ has no triangles, $\depth R/I^2 > 1$ if and only if $\diam \overline \H \le 2$.  
Since the induced subgraphs of $\overline \H$ on $F$ and $G$ are complete graphs, $\diam \overline H \le 2$ if and only if 
every vertex $i \in F$ is adjacent to a vertex in $G$ and every vertex $j \in G$ is adjacent to a vertex of $F$ in $\overline H$.
Translating this condition in terms of $\H$ we obtain the assertion.
\end{proof}

\begin{Example}
{\rm To illustrate  Theorem \ref{depth > 1 graph} we consider the following graphs:

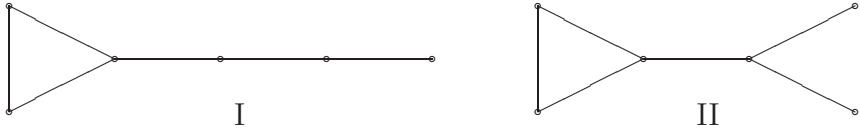
\begin{picture}(0,80)
\put(40,20){\line(2,1){40}}
\put(40,20){\line(0,1){40}}
\put(40,60){\line(2,-1){40}}
\put(80,40){\line(1,0){40}}
\put(120,40){\line(1,0){40}}
\put(160,40){\line(1,0){40}}
\put(40,20){\circle{2}}
\put(40,60){\circle{2}}
\put(80,40){\circle{2}}
\put(120,40){\circle{2}}
\put(160,40){\circle{2}}
\put(200,40){\circle{2}}
\put(125,15){I}

\put(240,20){\line(2,1){40}}
\put(240,20){\line(0,1){40}}
\put(240,60){\line(2,-1){40}}
\put(280,40){\line(1,0){40}}
\put(320,40){\line(2,1){40}}
\put(320,40){\line(2,-1){40}}
\put(240,20){\circle{2}}
\put(240,60){\circle{2}}
\put(280,40){\circle{2}}
\put(320,40){\circle{2}}
\put(360,60){\circle{2}}
\put(360,20){\circle{2}}
\put(300,15){II}
\end{picture}

\noindent Graph I does satisfy (a) but not (b), whereas graph II does satisfy (b) but not (a).
Since $\depth R/I^2 > 0$ by Theorem \ref{depth > 0 graph}, $\depth R/I^2 = 1$ in both cases.}
\end{Example}

\section{Cover ideals}

Let $\G$ be a hypergraph. The {\em cover ideal} of $\G$ is defined to be the intersection of all ideals $P_C$, $C \in \G$.
Let $\H(\G)$ denote the hypergraph of the covers of $\G$.  
It is easy to see that the cover ideal of $\G$ is the edge ideal of $\H(\G)$.
Therefore, we can use the results of the previous sections to study the associated primes and the depth of the second power of a cover ideal.
\par

Let $I$ be the cover ideal of $\G$. We need to describe the 2-saturating sets of covers of $\H(\G)$ 
in order to characterize the associated primes of $I^2$ and the depth of $R/I^2$. This can be done if $\G$ is a graph.
For that we need the following observation.

\begin{Lemma} \label{section}
Let $\G$ be a graph and $C$ a set of vertices.
The  hypergraph $\H(\G)_C$ and $\H(\G|_C)$ have the same minimal edges.
\end{Lemma}

\begin{proof} 
We only need to show that every edge of $\H(\G)_C$ resp. $\H(\G|_C)$ is contained in an edge of $\H(\G|_C)$ resp. $\H(\G)_C$.
Let $F$ be an arbitrary edge of $\H(\G)_C$. 
Then $F = H \cap C$ for some cover $H$ of $\G$.
Let $D$ denote the vertex set of the graph $\G|_C$. Then $F \cap D$ is a cover of $\G|_C$.
Hence $F$ contains an edge of $\H(\G|_C)$.  Let $G$ be an arbitrary edge of  $\H(\G|_C)$. 
Then $G$ meets every edge of $\G$ in $C$.
Hence $G \cup (V \setminus C)$ is a cover of $\G$. Since $(G \cup (V \setminus C)) \cap C = G$,
$G$ is an edge of $\H(\G)_C$. 
\end{proof}

In the following we call a set of vertices in a graph an induced cycle of the graph if it is the vertex set 
of an induced cycle.

\begin{Proposition} \label{odd}
Let $\H = \H(\G)$ for a graph $\G$. Let $C$ be a cover of $\H$.
A subset $U$ is a 2-saturating set of $C$ if and only if 
$|U| \le 1$ and $C$ is an edge or $U = C$ and $C$ is an induced odd cycle of $\G$.
\end{Proposition}

\begin{proof} 
 Let $|U| \le 1$. By Lemma \ref{minimal cover}, $U$ is a 2-saturating set of $C$ if and only if $C$ is a minimal cover of $\H$ 
or, equivalently, $C \in \G$. \par
Let $|U| > 2$. If $U$ is a 2-saturating set of $C$, then $C \not\in \G$. Hence $|C| \ge 3$.  
Assume that there exists $v \in C \setminus U$.
If $v$ is adjacent to a vertex $w \in C$, we choose a vertex $i \in C$, $i \neq v,w$.
Then $\{v,w\} \in \G|_{C \setminus i}$.
Since $U \setminus i$ is decomposable in $\H_C(i) = \H_{C \setminus i}$,
$U \setminus i$ contains  two disjoint covers of $\G|_{C \setminus i}$ by Lemma \ref{section}.
One of these covers must contain $v$, which implies $v \in U$, a contradiction.
If $v$ is not adjacent to any vertex of $C$, then $\G|_{C \setminus v} = \G|_C$.
Since $U \setminus v$ is decomposable in $\H_C(v) = \H_{C \setminus v}$, 
$U$ is decomposable in $\H(\G|_{C \setminus v})$ by Lemma \ref{section}. 
Since  $\H(\G|_{C \setminus v}) = \H(\G|_C)$, 
Lemma \ref{section} also shows that $U$ is decomposable in $\H_C$, a contradiction.
So we get $U=C$.  
Since $U$ is indecomposable in $\H_C$, $U$ is indecomposable in $\H(\G|_C)$.
Therefore, $\G|_C$ is not bipartite. Hence $\G|_C$ has an induced odd cycle $D$.
If $C$ is not an odd cycle, there is a vertex $i \in C \setminus D$. Then $D$ is an odd cycle in 
$\G|_{C \setminus i}$. This implies that $\G|_{C \setminus i}$ doesn't have two disjoint covers. 
Thus, $U \setminus i$ is indecomposable in $\H_{C \setminus i} = \H_C(i)$, a contradiction.
\par
Conversely, if $C$ is an odd cycle, it is easy to see that $C$ is indecomposable in $\H_C$. 
and $C \setminus i$ is decomposable in $\H_C(i)$. Hence $C$ is 2-saturating in $C$.
\end{proof}

From the above description of 2-saturating sets we immediately obtain the following combinatorial characterization 
of the associated primes of the second power of a cover ideal, which was obtained by Francisco-Ha-Van Tuyl  by different arguments.

\begin{Corollary} \label{FHV} \cite[Theorem 1.1]{FHV}
Let $I$ be the cover ideal of a graph $\G$. Let $C$ be a subset of $V$.
Then $P_C$ is an associated prime ideal of $I^2$ if and only if $C$ is an edge or an odd cycle of $\G$.
\end{Corollary}

\begin{proof} 
The assertion immediately follows from Theorem \ref{asso} and Proposition \ref{odd}.
\end{proof}

Furthermore, we are able to describe the complexes $\D_\a(I^2)$, 
which encode information on the depth of $R/I^2$.

\begin{Lemma} \label{facet 3}
Let $I$ be the cover ideal of a graph $\G$ and $\a \in \NN^n$. Then
$$\F(\D_\a(I^2)) = \{F|\ \overline F \in \G \text{ with } \overline F \not\subseteq H_\a 
\text{ or } \overline F \subseteq H_\a \text{ is an odd cycle of } \G\}.$$
\end{Lemma}

\begin{proof} 
By Lemma \ref{facet 2} we have
$$\F(\D_\a(I^2)) = \{F|\  H_\a \cap \overline F \text{ is a 2-saturating set of } \overline F\}.$$
Therefore using Proposition \ref{odd} we can see that $F \in \F(\D_\a(I^2))$ with $|H_\a \cap \overline F| \le 1$ if and only if $\overline F \in \G$ with $\overline F \not\subset H_\a$ and that $F \in \F(\D_\a(I^2))$ with $|H_\a \cap \overline F| > 1$ if and only if $\overline F \subseteq H_\a$ and $\overline F$ is an odd cycle of $\G$.
\end{proof}

To obtain a combinatorial criterion for $\depth R/I^2 > 1$ 
we need to know when there is a disconnected complex $\D_\a(I^2)$.

\begin{Lemma} \label{connect 2}
Let $I$ be the cover ideal of a graph $\G$ on $V =[n]$, $n \ge 4$.
Then there exists $\a \in \NN^n$ such that  $\D_\a(I^2)$ is disconnected if and only if one of the following conditions are satisfied:\par
{\rm (a)} $\G$ is the union of two disjoint edges or a path of length 3, \par
{\rm (b)} $\G$ has an induced odd cycle of length $n-1$ and the remained vertex is connected to this cycle,\par
{\rm (c)} $\G$ is the disjoint union of an induced odd cycle of length $n-2$ and an edge,\par
{\rm (d)} $\G$ has two induced odd cycles on vertex sets $C,D$ such that $V = C \cup D$. 
\end{Lemma}

\begin{proof} 
Assume that there exists $\a \in \NN^n$ such that  $\D_\a(I^2)$ is disconnected. 
Let $F$ and $G$ be two disjoint facets of $\D_\a(I^2)$.
Then $V = \overline F \cup \overline G.$
Therefore, $H_\a = (H_\a \cap \overline F) \cup (H_\a \cap \overline G)$. \par
 
If $\overline F, \overline G \in \G$, then $|V| = 4$. 
By Lemma \ref{minimal cover}, $|H_\a \cap \overline F| \le 1$ and $|H_\a \cap \overline G| \le 1$.
Therefore, $|H_\a| \le 2$. Hence $H_\a$ doesn't contain any odd cycle. By Lemma \ref{facet 3},
$$\F(\D_\a(I^2)) = \{F|\ \overline F \in \G \text{ with } \overline F \not\subseteq H_\a\}.$$
Since $\D_\a(I^2)$ is disconnected, $\F(\D_\a(I^2))$ must be the union of two disjoint edges. 
Let $\G^*$ be the graph of the complements of the edges of $\G$.
If  $H_\a$ is not an edge of $\G$,  $\F(\D_\a(I^2)) = \G^*$. 
Hence $\G$ must be the union of two disjoint edges. 
If  $H_\a $ is an edge of $\G$,  say $\{1,2\}$, then $\F(\D_\a(I^2)) = \G^* \setminus \{3,4\}$. 
Since  $\F(\D_\a(I^2)) $ is the union of two disjoint edges, $\G^*$ and therefore $\G$ is a path of length 3. 
So (a) is satisfied in this case.\par

If only one the sets $\overline F, \overline G$ is an edge of $\G$, say $\overline G$, then $|V| \le |\overline F|+2$. 
By Lemma \ref{facet 3}, $\overline F$ is an induced odd cycle of $\G$.
If $|V| = |\overline F| +1$, then $|\overline F| = n-1$ and the remained vertex is connected to $\overline F$ at least by $\overline G$.
Hence (b) is satisfied.
If $|V| = |\overline F| +2$, then $|\overline F| = n-2$ and $\overline F \cap \overline G = \emptyset$.
Hence $V = F \cup G$.  Since $\D_\a(I^2)$ is disconnected,  $\D_\a(I^2)$ has only two facets $F$ and $G$. Thus, 
$\G$ must be the union of the induced odd cycle on $\overline F$ and the edege $\overline G$.  Hence (c) is satisfied. \par
 
 If $\overline F, \overline G \not\in \G$, then $\overline F, \overline G$ are induced odd cycles of $\G$ by Lemma \ref{facet 3}.
 Since $V = \overline F \cup \overline G$, this implies (d).  This proves the necessary part of the assertion. \par
 
 Conversely, assume that one of the conditions (a) to (d) is satisfied.
 Using Lemma \ref{facet 3}, one can easily show that $\D_\a(I^2)$ is disconnected in the following cases:\par
 (a) $\G$ is the union of two edges $\{1,2\}, \{3,4\}$ and $\a = (1,0,1,0)$ or  $\G$ is the path $\{1,2\}, \{2,3\}, \{3,4\}$ 
and $\a = (0,1,1,0)$, \par
 (b') $\G$ has an induced odd cycle on the vertex set  $\{1,...,n-1\}$ 
and the vertex $n$ is connected to this cycle and $\a = (1,...,1,0)$,\par
 (c') $\G$ is the union of an induced odd cycle on the vertex set $\{1,...,n-2\}$
and the edge $\{n-1,n\}$ and $\a = (1,...,1,0,0)$,\par
 (d') $\G$ is the union of two induced odd cycles on vertex sets $C,D$ with $V = C \cup D$ and $\a = (1,...,1)$. \par
 This completes the proof of Lemma \ref{connect 2}.
\end{proof}

\begin{Theorem} \label{depth > 1 cover}
Let $I$ be the cover ideal of a graph $\G$ on the vertex set $V = [n]$, $n \ge 4$.
Then $\depth R/I^2 > 1$ if and only if the following conditions are satisfied:\par
{\rm (a)} $\G$ is not the union of two disjoint edges or a path of length 3, \par 
{\rm (b)} $\G$ has no induced odd cycle of length $n-1$, \par 
{\rm (c)} $\G$ is not a disjoint union of an induced odd cycle of length $n-2$ and an edge, \par
{\rm (d)} $\G$ has no pair of induced odd cycles with vertex set $C,D$ such that $V = C \cup D$,\par
{\rm (e)} $\G$ is not an odd cycle.
\end{Theorem}

\begin{proof} 
By Corollary \ref{FHV}, $\depth R/I^2 > 0$ if and only if (e)  is satisfied.
Since $\depth R/I^2 = \min\{i|\ H_\mm^i(R/I^2) \neq 0\}$, 
it suffices to show that $H_\mm^1(R/I^2) = 0$ if and only if (a) to (d) are satisfied. 
By Proposition \ref{h1}, $H_\mm^1(R/I^2) = 0$  if and only if 
$\D_\a(I^2)$ is connected  for all $\a \in \NN^n$ and $\depth R_j/I_j> 0$ for all $j \in V$, 
where $R_j = k[x_i| i \neq j]$ and $I_j = IR[x_j^{-1}]\cap R_j$. 
Let $\H$ denote hypergraph of the covers of $\G$. Then $I_j$ is the edge ideal of
$\H_{V \setminus j}$. By Lemma \ref{section}, $I_j$ is also the cover ideal of the graph $\G|_{V \setminus j}$.
Hence $\depth R_j/I_j> 0$ if and only if $V \setminus j$ is not an odd cycle by Corollary \ref{FHV}. 
This condition holds for all $j \in V$ if and only if $\G$ has no induced odd cycle of length $n-1$.
Therefore, we only need to show that $\D_\a(I^2)$ is connected for all $\a \in \NN^n$ if and only if (a) to (e) are satisfied. 
But this follows from Lemma \ref{connect 2}.
\end{proof}

Recall that a ring $S$ satisfies Serre's condition $(S_2)$ if $\depth S_P \ge \min\{2,\height P\}$ for every prime ideal $P$ of $S$.
Rinaldo, Terai and Yoshida asked whether for  a squarefree monomial ideal $I$, $R/I^2$ is Cohen-Macaulay if $R/I^2$ 
satisfies $(S_2)$ \cite[Question 3.1]{RTY1}. 
If $I$ is a cover ideal, we give a positive answer to this question by proving the following stronger implication. 

\begin{Theorem} \label{S2 cover}
Let $I$ be the cover ideal of a graph $\G$. Then $R/I^2$ satisfies Serre's condition $(S_2)$ if and only if
$I$ is a complete intersection.
\end{Theorem}

\begin{proof} 
It is well-known that $R/I^2$ is Cohen-Macaulay if $I$ is a complete intersection.
Therefore, it remains to show that if $R/I^2$ satisfies $(S_2)$, then $I$ is a complete intersection.
If $n = 2$, then $I$ is a complete intersection.
If $n = 3$, then $R/I^2$ satisfies $(S_2)$ if $\depth R/I^2 > 0$. 
By Corollary \ref{FHV}, this condition is satisfied if $\G$ is not a 3-cycle. 
Hence $\G$ is a path of length 2 or 3, which implies that $I$ is a complete intersection. \par
 Assume that $n \ge  4$ and $R/I^2$ satisfies $(S_2)$.
Then $I^2$ has no embedded associated prime. 
Hence $\G$ has no induced odd cycle by Corollary \ref{FHV}.
This means that $\G$ is a bipartite graph. 
We prove now that every pair of disjoint edges $F,G$ of $\G$ is contained in a 4-cycle.  
Set $S = k[x_i|\ i \in F \cup G]$ and $J = IR[x_i^{-1}|\ i \not\in F \cup G]\cap R$.
Let $\H$ be the hypergraph of the covers of $\G$. Then $J$ is the edge ideal of $\H_{F \cup G}$.
By Lemma \ref{section}, $J$ is also the cover ideal of the graph $\G' := \G|_{F \cup H}$. 
Let $\D$ and $\D'$ be the simplicial complexes whose facets are the complements of the edges of $\G$ and $\G'$ in their vertex sets, respectively. Then $I$ and $J$ are the Stanley-Reisner ideals of $\D$ and $\D'$.
Moreover, $\D'$ is the link of $V \setminus (F \cup G)$. Therefore, by \cite[Lemma 4.2]{TT}, 
$S/J^2$ also satisfies $(S_2)$.
Since $\G'$ is a graph on 4 vertices, this implies $\depth S/J^2 = 2$.
By Theorem \ref{depth > 1 cover} (a) and (e), $\G'$ must be a 4-cycle.
Thus, every pair of disjoint edges of $\G$ is contained in a 4-cycle.  
Since $\G$ is a bipartite graph, this implies that $\G$ is a complete bipartite graph.
Hence $I$ is a complete intersection.
\end{proof}

As examples we compute the depth of $R/I^2$ for the cover ideal of all graphs of $4,5$ vertices.

\begin{Example}
{\rm For $n = 4$ we have  $\dim R/I^2 = 2$.  
By Corollary \ref{FHV}, $\depth R/I^2 > 0$ because the graph can't be an odd cycle.
By Theorem \ref{S2 cover}, $R/I^2$ is Cohen-Macaulay, if and only if $I$ is a complete intersection.
This means that the graph is complete bipartite. Therefore, $\depth R/I^2 = 1$ except
the following two cases, where $\depth R/I^2 = 2$.

\begin{picture}(0,75)
\put(100,20){\line(5,2){50}}
\put(100,60){\line(5,-2){50}}
\put(100,40){\line(1,0){50}}
\put(100,20){\circle{2}}
\put(100,60){\circle{2}}
\put(100,40){\circle{2}}
\put(150,40){\circle{2}}

\put(250,20){\line(1,0){40}}
\put(250,20){\line(0,1){40}}
\put(250,60){\line(1,0){40}}
\put(290,20){\line(0,1){40}}
\put(250,20){\circle{2}}
\put(250,60){\circle{2}}
\put(290,20){\circle{2}}
\put(290,60){\circle{2}}
\end{picture}

For $n = 5$ we have $\dim R/I^2 = 3$. By Corollary \ref{FHV}, $\depth R/I^2 = 0$ if and only if the graph is a 5-cycle.
By Theorem \ref{depth > 1 cover},  $\depth R/I^2 = 1$ if and only if the graph is a disjoint union of a triangle and an edge or the union of two triangles meeting at a vertex:

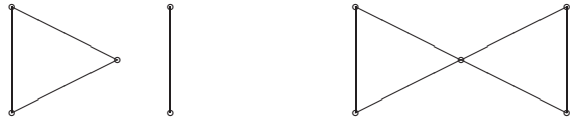
\begin{picture}(0,75)
\put(100,20){\line(2,1){40}}
\put(100,60){\line(2,-1){40}}
\put(100,20){\line(0,1){40}}
\put(160,20){\line(0,1){40}}
\put(100,20){\circle{2}}
\put(100,60){\circle{2}}
\put(140,40){\circle{2}}
\put(160,20){\circle{2}}
\put(160,60){\circle{2}}

\put(230,20){\line(2,1){40}}
\put(230,20){\line(0,1){40}}
\put(230,60){\line(2,-1){40}}
\put(310,20){\line(-2,1){40}}
\put(310,20){\line(0,1){40}}
\put(310,60){\line(-2,-1){40}}
\put(230,20){\circle{2}}
\put(230,60){\circle{2}}
\put(270,40){\circle{2}}
\put(310,20){\circle{2}}
\put(310,60){\circle{2}}
\end{picture}

\noindent In all other cases, $\depth R/I^2 = 2$ except the following two cases when  the graph is complete bipartite and $\depth R/I^2 = 3$  by Theorem \ref{S2 cover}.}

\begin{picture}(0,90)
\put(100,20){\line(5,3){50}}
\put(100,60){\line(5,-1){50}}
\put(100,80){\line(5,-3){50}}
\put(100,40){\line(5,1){50}}
\put(100,20){\circle{2}}
\put(100,40){\circle{2}}
\put(100,60){\circle{2}}
\put(100,80){\circle{2}}
\put(150,50){\circle{2}}

\put(240,30){\line(5,1){50}}
\put(240,30){\line(5,3){50}}
\put(240,50){\line(5,-1){50}}
\put(240,50){\line(5,1){50}}
\put(240,70){\line(5,-1){50}}
\put(240,70){\line(5,-3){50}}
\put(240,30){\circle{2}}
\put(240,50){\circle{2}}
\put(240,70){\circle{2}}
\put(290,40){\circle{2}}
\put(290,60){\circle{2}}
\end{picture}
\end{Example}

\end{document}